\documentclass{amsart}

\usepackage{amsmath, amsthm, amssymb, amsfonts}
\usepackage[dvips]{graphicx}
\usepackage{overpic}
\usepackage{enumitem}
\usepackage[colorlinks=true]{hyperref}

\renewcommand{\emph}{\textbf}
\usepackage[margin=3cm]{geometry}

\newtheorem{theorem}{Theorem}[section]
\newtheorem{corollary}[theorem]{Corollary}
\newtheorem{proposition}[theorem]{Proposition}
\newtheorem{definition}[theorem]{Definition}
\newtheorem{lemma}[theorem]{Lemma}

\newtheorem{remark}[theorem]{Remark}

\newtheorem*{theorem*}{Theorem}
\newtheorem*{proposition*}{Proposition}
\newtheorem*{definition*}{Definition}
\newtheorem*{lemma*}{Lemma}
\newtheorem*{claim*}{Claim}
\newtheorem*{corollary*}{Corollary}
\newtheorem*{convention*}{Convention}

\theoremstyle{definition}

\newtheorem{construction}[theorem]{Construction}

\theoremstyle{remark}

\newtheorem*{rem*}{Remark}

\DeclareMathOperator{\ind}{ind}
\DeclareMathOperator{\MCG}{MCG}

\DeclareMathOperator{\PMod}{PMod}

\title{New Anosov flows via bicontact structures}
\author{Tali Pinsky}
\address{The Technion, Haifa and Monash University, Melbourne}
\email{talipi@technion.ac.il}
\author{Federico Salmoiraghi}
\address{Saint Louis University, Saint Louis}
\email{federico.salmoiraghi@gmail.com}
\thanks{The authors were supported by the Israel Science Foundation (grant No. 51/4051).}
\date{\today}

\begin{document}

\begin{abstract}

We present a new approach to hyperbolic plugs, via a construction of bicontact plugs on 3-manifolds with boundary that are surface bundles over the circle. The boundary components are quasi transverse tori, and we prove a gluing theorem that allows us to produce closed manifolds carrying new transitive Anosov flows. We show that a toroidal manifold produced by gluing two copies of the figure eight knot complement may carry many nonequivalent Anosov flows, and likewise a manifold composed of a figure eight complement and a trefoil complement.
We further show that certain generalized Handel--Thurston surgeries can be realized as sequences of Goodman--Fried surgeries and produce new examples of different surgery sequences resulting in the same Anosov flow.
\end{abstract}

\maketitle

\section{Introduction}

Anosov flows represent an important class of dynamical systems with remarkable ergodic and geometric properties~\cite{Ano1}. These are among the simplest examples of chaotic dynamical systems.
The prototype of an Anosov flow is the geodesic flow on the unit tangent bundle of a hyperbolic surface; another classical example is the suspension flow of an Anosov diffeomorphism of the torus. 
For nearly two decades, these two types were the only known examples.

The introduction of cut-and-paste techniques enabled the construction of many new examples. The first non-transitive Anosov flow was constructed by Franks and Williams~\cite{FranksWilliams1979} by gluing two figure eight complements along their boundaries, where the boundaries are transverse to the flows. Handel and Thurston~\cite{HandelThurston1980} constructed new transitive Anosov flows by gluing pieces of a geodesic flow. In this case the boundary is quasi transverse, i.e. contains a finite number of tangent orbits and transverse annuli between them where the flow points to different directions, yet the pieces glue to a manifold carrying a smooth Anosov flow.
Goodman~\cite{Goo} introduced the first surgery of Dehn type on Anosov flows, interpreting a closed orbit as a knot, and used it to construct the first examples of Anosov flows on hyperbolic manifolds.

More recently, Bonatti, B\'{e}guin, and Yu \cite{BeguinBonattiYu2017} developed a systematic method for constructing Anosov flows by gluing \emph{hyperbolic plugs} along their transverse boundaries. They produced the first examples of closed 3-manifolds supporting at least $n$ pairwise non-orbit-equivalent Anosov flows for any fixed $n \in \mathbb{N}$. Bowden and Mann \cite{BowdenMann2022} construct hyperbolic manifolds supporting many Anosov flows. Clay and Pinsky \cite{ClayPinsky2025Graph} Construct graph manifolds consisting of two copies of the trefoil complement that carry many transitive Anosov flows, by gluing covers of geodesic flows along quasi-transverse boundaries. Paulet \cite{Paulet2025} proved a general theorem that allows to construct new flows from building blocks with quasi-transverse boundaries.

These results all rely on the study of the invariant foliations of the flow. Perhaps surprisingly, Anosov flows are also intimately related to contact structures, i.e. maximally non-integrable plane fields. Mitsumatsu \cite{Mitsumatu1995} first observed that a vector field defining an Anosov flow can always be viewed as the intersection of a pair of transverse contact structures, one positive and one negative. Not every such bicontact structure defines an Anosov flow. However, Hoozori recently proved a converse theorem: a flow is Anosov if and only if it is defined by a \emph{strongly adapted} bicontact structure  i.e., one satisfying the condition that the Reeb flow of one structure is contained in the plane field of other structure in the pair (see \cite{Hoz5}). The Reeb vector field is called the \emph{Legendrian Reeb} field.
We call the flow defined by the intersections of the two plane fields the \emph{supported} flow.

The goal of the present work is to use the relation between Anosov flows and contact geometry to introduce a method for constructing building blocks with controlled geometry, and to use them to produce new examples of transitive Anosov flows. We say a vector field is \emph{periodic} on $\partial M$ if it is tangent to the boundary, and its orbits on the boundary are all periodic.

\begin{theorem}[Gluing theorem]
\label{thm: gluing plugs v1}
Let $\mathcal{M}_1$ and $\mathcal{M}_2$ be $3$-manifolds with boundary, carrying strongly adapted bicontact structures. If the Legendrian Reeb vector fields are periodic on two components $T_1\subset\partial\mathcal{M}_1$ and $T_2\subset\partial\mathcal{M}_2$ and the number of tangent orbits of the flow defined by the intersection is equal, then they can be glued, yielding a strongly adapted bicontact structure on $\mathcal{M} = \mathcal{M}_1 \cup \mathcal{M}_2$.
\end{theorem}

Thus, when we glue all boundary components we obtain a closed manifold carrying an Anosov flow by \cite{Hoz5}.
In particular, no explicit cone criterion for hyperbolicity is necessary.\\

The Fried Conjecture states that any transitive Anosov flow is obtained through a finite sequence of Goodman-Fried surgeries from the suspension flow of a hyperbolic linear automorphism of the torus (i.e. it is {\it almost orbit equivalent} to a suspension) \cite{Fri}.
Dehornoy and Shannon~\cite{DeSh19} showed that the geodesic flow of a negatively curved orbifold is almost orbit equivalent to a suspension flow of a hyperbolic linear automorphism of the torus. 
Tsang~\cite{Tsang24} established this for some further cases. We use our framework to construct new examples of Anosov flows on toroidal manifolds, and show that they are almost equivalent to geodesic flows. 

The building blocks we use are surface bundles over the circle, where the fiber $\Sigma_{g,b}$ is a genus $g$ surface with $b$ boundary components (these are three manifolds with toroidal boundaries).
Let $V$ be a vector field with $b$ non-positive index singularities on a surface $\Sigma_g$. Consider a finite family $C=\{c_1,\cdots ,c_r\}$ of non-separating curves on $\Sigma_{g,b}$ that are transverse to $V$, and let $\Lambda_C$ be the subgroup  of $\MCG(\Sigma_{g,b})$ generated by Dehn twist along the curves in $C$.

\begin{theorem}[Generating bundle plugs]\label{thm:bundle plugs}
For every $f\in \Lambda_C$ the mapping torus $\mathcal{M}_f$ with fiber $\Sigma_{g,b}$ and monodromy $f$ carries a bicontact structure so that $f$ is the first return map of a Legendrian Reeb vector field. Furthermore, the supported flows can be embedded in an Anosov flow on a closed manifold, and are almost orbit equivalent to geodesic flows.
\end{theorem}

The bicontact structure is  explicitly constructed as a union of a vertically invariant contact structure $\xi_-$ approximating the fiber $\Sigma_{g,b}$, and a vertically rotating contact structure $\xi_+$. Surprisingly, we can allow the structure $\xi_+$ to rotate more than once, obtaining infinitely many bicontact structures.

If the surface $\Sigma_{g,b}$ is a torus with $b$ boundary component $T^2_b$ there is a simple vector field $V$ such that there is a complete set $\{c_1,\cdots, c_{b+1}\}$ of curves so that Dehn twists on them generate the pure mapping class group, and they are all transverse to $V$. Thus we obtain the following.

\begin{proposition}
[Punctured torus bundles]
\label{pro:torus}
Let  $T^2_b$ denote the $b$-times punctured torus. Every $T^2_b$-bundle over the circle admits infinitely many  strongly adapted bicontact structures $\mathcal{B}^k = (\alpha_-^k, \xi_+^k),\ k\in\mathbb{N}$ such that $\xi_+^n$ is isotopic to $\xi_+^m$ if and only if $n = m$. Each boundary component is quasi-transverse for the supported flow $X^t$ and periodic for $R_-^k$.
\end{proposition}

Recently Bonatti, Béguin, Ma and Yu \cite{BBMY}  have also constructed transitive Anosov flows on closed fibered hyperbolic $3$-manifolds. Their construction uses Goodman--Fried surgery along orbits parallel to the fibers of surface bundles. 

Yang and Yu ask in \cite{YangYu2022Classifying} whether a toroidal manifold composed of two figure-eight knot complement pieces can carry a transitive Anosov flow. We show that in fact it can carry many.

\begin{corollary}
\label{cor:figure8}
For any natural number $k$, there exists a closed toroidal manifold supporting at least $k$ pairwise non-orbit-equivalent transitive Anosov flows, obtained by gluing two copies of the figure-eight knot complement $M_8$.
\end{corollary}

Using our construction we are also able to generalize the Handel--Thurston  and Clay--Pinsky \cite{ClayPinsky2025Graph} examples. Namely, we are able to glue any two covers of geodesic flows on a manifold with boundary that have the same number of periodic orbits. Furthermore, we are able to prove that flows obtained by gluing covers of geodesic flows on once-punctured tori are almost equivalent to geodesic flows.\\
 
\textbf{Acknowledgments.} The authors wishes to thank Thomas Barthelmé, Chi-Cheuk Tsang, Neige Paulet, Surena Hozoori and Antonio Alfieri for illuminating conversations and helpful comments.

\section{Anosov flows and bicontact strucures}

\textbf{Contact structures and Reeb flows.} A co-oriented \textbf{contact structure} is a plane field distribution that is maximally non-integrable in the sense that it can be described as the kernel of a $C^1$ 1-form satisfying the relation $\alpha\wedge d\alpha\neq0$. The condition of contactness ensures that there is not a subsurface $S$ such $TS=\ker \alpha$, even in a neighborhood of a point. Thus, contact structures can be thought as polar opposite of foliations. Contact structures are either positive or  negative. A positive contact structure is a plane field distribution $\xi_+$ described by a $C^1$ 1-form satisfying the relation $\alpha_+\wedge d\alpha_+>0$. A negative contact structure $\xi_-$ is described instead by a $C^1$ 1-form such that $\alpha_-\wedge d\alpha_-<0$. A contact form $\alpha$ such that $\ker \alpha=\xi$ is not unique: in general every contact form $\alpha'$ supporting $\xi$ is of the form $\alpha'=f\alpha$ for some non vanishing function $f:M\rightarrow \mathbb{R}$.

Note that contact structures may be defined on manifolds with boundary. Let $(M,\xi)$ be a contact manifold with boundary $\partial M$. The contact plane distribution $\xi$ induces on $\partial M$ a well defined singular 1-dimensional foliation called the \textbf{characteristic foliation} induced by $\xi$ on $\partial M$. This fact is not in general true for plane fields that are not maximally non-integrable.

An important property of contact structures is that they do not have local invariants. Indeed by Darboux Theorem positive (negative) contact structures are all locally contactomorphic to the positive (negative) {\it standard contact structure in} $\mathbb{R} ^3$ described by $\xi_{std}^+=\ker dz-y\:dx$ and $\xi_{std}^-=\ker dz+y\:dx$. Therefore we can locally visualize a positive contact (negative) structure as a plane field whose plane rotate counterclockwise (clockwise) along the $x$-axis. 

Suppose that $\alpha$ is a contact form. There is an important flow associated to $\alpha$ called the \textbf{Reeb flow}. This is the unique vector field $R_{\alpha}$ satisfying the equations
\[
\alpha(R_{\alpha})=1,\,d\alpha(R_{\alpha},\cdot)=0.
\]
These relations imply that $R_{\alpha}$ is transverse to $\ker \alpha$ and $\mathcal{L}_{R_{\alpha}}\alpha=0$ (i.e. $R_{\alpha}$ preserves $\alpha$). Moreover, the Reeb vector field preserves the natural volume form $\alpha\wedge d\alpha$ called the {\it Liouville} form (see \cite{Et04}).
\\

\begin{figure}[!ht]
    \centering
    \includegraphics[width=0.5\textwidth]{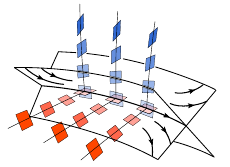}
    \caption{A Bicontact structure and the invariant foliations along a flowline of an Anosov flow.}
    \label{fig:F1}
\end{figure}

Mitsumatsu \cite{Mitsumatu1995} first noticed that the generating vector field of an Anosov flow also belongs to the intersection of two transverse contact structures $(\xi_-,\xi_+)$, one negative and one positive (see also Eliashberg-Thurston \cite{ElTh}) called \textbf{bicontact structure}, as depicted in Figure~\ref{fig:F1}. Such a pair of plane fields will be denoted with the symbol $\mathcal{B}$ in the rest of the paper. However, not all bi-contact structures define Anosov flows. For instance it is possible to show that there are bicontact structures on $S^3$ and $T^3$ but it is known that neither $S^3$ nor $T^3$ support an Anosov flow.
Hoozori gives a characterization of bicontact structures defining Anosov flows. 

\begin{theorem}[Hozoori \cite{Hoz5}]
\label{thm:Hoz}
A flow is Anosov if and only if it lies in the intersection of a pair $(\alpha_-,\xi_+)$ where $\xi_+$ is positive contact structure and $\alpha_-$ is a negative contact form, such that the Reeb vector field of $\alpha_-$ is contained in $\xi_+$.
\end{theorem}

A bicontact structure $(\xi_-,\xi_+)$ that is supported by a pair $(\alpha_-,\xi_+)$ where $\alpha_-$ is a contact form and $\xi_+$ is a contact structure such that $R_-$ is contained in $\xi_+$ as depicted in Figure~\ref{fig:F7}, is called $\textbf{strongly adapted}$ (abbreviation $\mathcal{SA}$). 

\begin{figure}
    \centering
    \includegraphics[width=0.4\textwidth]{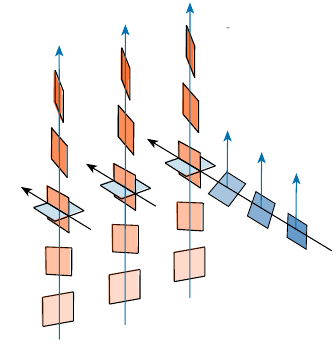}
    \caption{Strongly adapted bicontact structure. The blue vertical arrows represent the Reeb $R_-$ vector field of the negative contact form $\alpha_-$. The supported vector field is represented by the black arrows.}
    \label{fig:F7}
\end{figure}

\begin{remark}
    There is a one to one correspondence between strongly adapted bicontact structures and Legendrian Reeb vector fields, i.e. Reeb vector fields for a negative (or positive) contact form that are Legendrian for a positive (or negative) contact structure. 
\end{remark}

\section{Gluing bicontact plugs and bicontact surgeries}

The goal of this chapter is to present a gluing theorem for three manifolds with boundaries endowed with a strongly adapted  bicontact structure with a particular type of boundary behavior.

\begin{definition}
A \textbf{bicontact plug} $\mathcal{M}$ is a pair $(M,\mathcal{B})$ where $M$ is a manifold with boundary equipped with a bi-contact structure $\mathcal{B}=(\xi_-,\xi_+)$. If the bicontact structure $\mathcal{B}$ is supported by a strongly adapted bicontact pair $(\alpha_-,\xi_+)$ we call $\mathcal{M}$ strongly adapted bicontact plug (abbreviation $\mathcal{SAB}$-plug).
\end{definition}

In the rest of the paper we denote a vector field that belongs to the intersection of $\xi_-$ and $\xi_+$ with the symbol $X$ and we will use $X^t$ to denote its flow. Moreover we will say that $X$ or $X^t$ are {\it supported} by the bicontact structure $\mathcal{B}$. We restrict our attention to $\mathcal{SAB}$-plug with particular boundary behaviours.

\begin{definition}A connected component $B$ of $\partial M$ is called {\emph{quasi-transverse periodic}} boundary component (abbreviation $\mathcal{QTP}$-boundary component) if

 \begin{enumerate}
  \item $B$ is a quasi-transverse torus for $X^t$ with $2n$ orbits.
  \item 
  The Reeb vector field $R_-$ of $\alpha_-$ near the boundary is periodic.
\end{enumerate}
\end{definition}

\begin{remark}
  Examples of $\mathcal{SAB}$-plugs with $\mathcal{QTP}$-boundary are the $k$-fold covers of the geodesic flows on unit tangent bundles of hyperbolic surfaces or orbifolds with geodesic boundary. The ambient manifolds are trivial bundles (in the surface case) or Seifert fiber spaces (see \cite{MontesinosBook87,ClayPinsky2025Graph}).
\end{remark}

 \begin{definition} 
    Let $\mathcal{M}_1$ and $\mathcal{M}_2$ be two $\mathcal{SAB}$-plugs with $\mathcal{QTP}$-boundary. We say that two components $B_1$ of $\partial \mathcal{M}_1$ and $B_2$ of $\partial \mathcal{M}_2$ are \textbf{compatible} if $B_1$ and $B_2$ have the same number of closed orbits.
\end{definition}

It turns out that we can glue plugs along compatible boundary components. The following is a more exact reformulation of Theorem~\ref{thm: gluing plugs v1}.

\begin{theorem}\label{thm: gluing plugs}

Let $\mathcal{M}_1$ and $\mathcal{M}_2$ be two $\mathcal{SAB}$-plugs with compatible $\mathcal{QTP}$-boundary components $B_1$ and $B_2$. After a perturbation of the bicontact structures in a collar neighborhood of $B_1$ an $B_2$, there is a diffeomorphism $F:B_1\rightarrow B_2$ such that $\mathcal{M}_1 \cup_F \mathcal{M}_2$ is a $\mathcal{SAB}$-plug. In particular, if $M=\mathcal{M}_1 \cup_F \mathcal{M}_2$ is closed, it carries an Anosov flow.
\end{theorem}

The proof of this statement is divided in two parts. First we prove that in the hypothesis of Theorem \ref{thm: gluing plugs} we can modify the bicontact structure near $B$ preserving the strongly adapted property in such a way that the new $2n$ closed orbits of $X^t$ on $B$ intersect the orbits of the new $R_-$ exactly once. This modification uses a construction called \textbf{bicontact surgery}. Then, the flexibility of the contact structure $\xi_+$ allows us to construct an isotopy from the initial bicontact structure to a standard one near the boundary. Since this isotopy can be achieved through strongly adapted bicontact structures it induces a one parameter family of flows. 

The second author of this paper introduced a $(1,q)$-Dehn type surgery in a strongly adapted bicontact structure that yields a new strongly adapted bicontact structure (see \cite{Sal1} and \cite{Sal2}). When the operation is performed near a biLegendrian, i.e. a closed orbit of the flow supported by the bicontact structure, it can be interpreted as the counterpart of Goodman-Fried surgery in the context of strongly adapted bicontact structures.

\begin{definition}
A knot $K$ in a manifold carrying a contact structure is \emph{Legendrian} if the tangent $TK$ is contained in the contact plane field. If the manifold carries a bicontact structure $(\xi_-,\xi_+)$, a knot $K$ is \textbf{Legendrian-transverse} (abbreviation L-t knot) if $TK$ is contained in $\xi_-$ and it is transverse to $\xi_+$.
\end{definition}

Note that given a closed orbit $\gamma$ of a flow $X^t$ supported by a strongly adapted bi-contact structure,
pushing it along one of the Reeb vector fields keeps is tangent to the corresponding contact structure, while making it transverse to the other structure.
Thus, a periodic orbit is always isotopic to a Legendrian-transverse knot $K$. 

\begin{definition}
    A \textbf{L-t push-off} of a closed orbit $\gamma$ in a strongly adapted bicontact structure is the L-t knot obtained by translating $\gamma$ using the flow of $R_-$ (see Figure \ref{fig:Surgery} on the right). 
\end{definition}

 \begin{theorem}[Salmoiraghi \cite{Sal1}]
     Let $K$ be a knot that is a Legendrian-transverse push-off of a closed orbit $\gamma$ in a strongly adapted bicontact structure $(\alpha_-,\xi_+)$. A Legendrian-transverse $(1,q)$-surgery yields a new strongly adapted bicontact structure for all $q\in \mathbb{Z}$. so that the resulting flow on $M'$ is orbit equivalent to a flow obtained by Goodman--Fried surgery along $K$.
 \end{theorem}
 \begin{remark}
    The bicontact surgery preserves more structure compared to Goodman--Fried surgery. Anosovity is a consequence of being supported by a pair $(\alpha_-,\xi_+)$ with $R_-$ Legendrian Reeb.  
\end{remark}

 \begin{figure}[ht!]
     \centering
     \includegraphics[width=12cm]{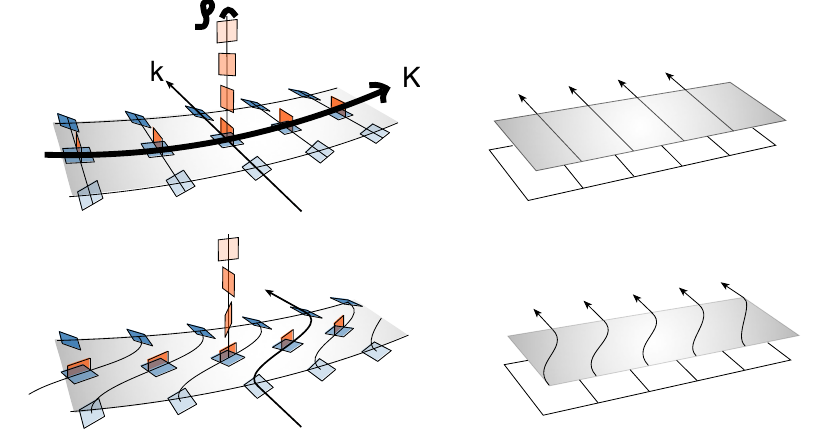}
     \caption{The action of the bicontact surgery on a neighborhood of the surgery annulus. In blue is the plane field corresponding to $
     \alpha_-$ and in red the plane field of $\xi_+$.}
     \label{fig:surgery}
 \end{figure}

{
The surgery acts in the following way, as depicted in Figure~\ref{fig:surgery} Open up the annulus into two parallel annuli sharing the boundary, and then glue the top to bottom in a different way. Let $s\times v\times w$ be coordinates in the original manifold, where $s$ is a parameter describing the core $K$ of the annulus and is glued to the core of the top annulus the same way as before the surgery. The arc $v$ in the bottom original annulus, along a flow-line of the Anosov flow connecting one boundary component to the other, is glued after the surgery the surgery to $v'=v+q\cdot K$ in the top annulus. As a result, half an arc parallel to $w$ below the annulus, which is transverse to the annulus and tangent to the Reeb field $R_-$ of the contact form $\alpha_-$ (blue in the figure) is cut open at the point where it hits the annulus, say $rv+\theta K$ where $0\leq r\leq 1$ and $s\in[0,2\pi]$, and is glued to another half $\rho$ arc starting upwards from $w+f(r)q\cdot K$ where $f(r)$ is a monotone function. The surgery topologically consists of inserting a disk bounded by the difference between flow-lines on the bottom and flow-lines on the top producing well defined flow-lines on a new manifold. 

In \cite{Sal1} it is proven that there is a deformation of the positive contact structure, so that the new
$\rho'$ arcs are still contained in the new red plane field $ker(\alpha_+)$ after the surgery, and so this results in a new strongly adapted bicontact field.
}

 \begin{remark}
     The bicontact surgery construction can be performed also along L-t knots on the boundary, as follows: If $K$ is an L-t closed curve on the boundary of $M$, since the Reeb flow $R_-$ is tangent to the boundary, and $\gamma$ is transverse to $\xi_+$, the Anosov flow-lines through $K$ must also be transverse to the boundary. In this case, only half of the surgery annulus that is tangent to the flow-lines as above is contained in $M$. Denote the half annulus by $A'$, then $\partial^+A'=K\in\partial M$. The surgery involves opening $A'$ to a wedge of two annuli joined along $\partial^- A'$, and then re-gluing the two half annuli so that a flowline $v$ in the ``bottom" boundary of the wedge will be glued to $v+q\cdot K$ with on the ``top" boundary. Note that we may allow surgeries with non-integer coefficients $q\in \mathbb{Q}$ for a boundary surgery.

 \end{remark}

 \begin{proposition}\label{pro:one intersection}
 Let $\gamma$ be a periodic orbit of a $\mathcal{QTP}$-boundary component of a $\mathcal{SAB}$-plug.
 There exists a rational surgery coefficient, so that the bicontact surgery along the L-t-push off of $\gamma$ results in an $\mathcal{SAB}$-plug for which $\gamma$ is an orbit of the Anosov flow, intersecting a new periodic Reeb orbit at a single point.
 \end{proposition}

\begin{proof}
    The assumption that the orbit $\gamma$ is primitive, implies that it is a single curve passing through its $q$ intersection points with a single orbit $w$ of the Reeb flow. 
    
    Consider a regular annular neighborhood $\gamma\times[-\varepsilon,\varepsilon]$ of $\gamma$ within the boundary torus, where we fix the framing so that the $[-\varepsilon,\varepsilon]$ are parallel to the direction of the Reeb flow. In particular, the closed curve $w$  intersects the neighborhood along $q$ vertical segments, connecting the bottom boundary to the top one (where bottom to top coincide with the direction of the Reeb flow).
    
    Starting from a point $A$ of $w$ on the top boundary of the neighborhood and tracing $w$ upwards, it continues into the boundary torus, until it intersects the neighborhood again along a point $B$ in its bottom boundary, $s$ segments (or intersection points) further along $\gamma$.
    
    Define the surgery to take the vertical $w$ segment into a diagonal segment with slope $-s/q$. This implies that after the surgery there is a single segment of the curve $\tilde w$ connecting the point $B$ to the point $A$ within the neighborhood, while in the complement of the neighborhood the Reeb flow did not change and so the segment of $w$ connecting $A$ back to $B$ through the complement is also an orbit of the new Reeb flow.

    Thus, this surgery results in a new $\mathcal{SAB}$-plug with $\tilde w$ the periodic Reeb orbit intersecting $\gamma$ once as required.
    Note that the original Reeb orbit $w$ is divided into $q$ new Reeb orbits by this surgery.
\end{proof}

\begin{remark}
Let $q$ be the number of intersections between a boundary closed orbit $\gamma$ and a periodic orbit $w$ of $R_-$. After a $\frac{q-1}{q}$ surgery with $q=1+k$ bicontact surgery along a L-t push-off of $\gamma$ the new periodic orbit $\tilde{w}$ of $\tilde{R}_-$ intersects $\gamma$ once. Moreover the trajectories $\tilde{w}$ of the new Reeb vector field $R_-$ on the boundary belong to the class $[\tilde{w}]=[w]+(q-1)[\theta]$ (or equivalently $[\tilde{w}]=[w]+k[\theta]$). A surgery that modifies $R_-$ in such a way that $\tilde{R}_-$ intersects $\gamma$ once is not unique. Every surgery coefficient $\frac{q-1}{q}+h$ where $h\in \mathbb{Z}$ achieve a single point intersection.
\end{remark}

The following result is a generalization of Lemma 5.4 of \cite{Sal1} and it affirms that after an isotopy of $\xi_+$ that induces a path of strongly adapted bicontact structures, we can assume that near the boundary the bicontact structure takes a canonical form.
\begin{lemma}[Canonical 1-forms]
Let $B$ be a periodic boundary component with $2n$ periodic orbits and such that the orbits of $R_-$ intersect a closed orbit $\gamma$ exactly once. Consider in a collar neighborhood  $N(B)$ the coordinate system $(s,v,w)$ where $w$ is the parameter describing the flow-lines of $R_-$, $s$ is the parameter describing a closed orbit $\gamma$, where $s\in [0,2\pi]$ is pointing to the right of $R_-$, i.e. $w\times s$ is pointing outwards through the boundary and the $v$-curves $v\in [-\delta,0]$ are transverse to $B$, so that $v=0$ on $B$.
In these coordinates, the negative contact structure is negatively invariant, and has the form $$\alpha_-=dw+v\:ds.$$ 

Moreover, after an isotopy of $\xi_+$ that induces a path of orbit equivalent flows we can assume that the positive contact structure is described as the kernel of  $$\alpha_+=
\sin(n w)\:ds+\cos(n w)\:dv,$$ where $w$ is parametrized to go from 0 to $2\pi$. Note that the $n$ tangency points along a single Reeb orbit correspond each to exactly one of the $n$ tangent orbits of the supported flow.
\end{lemma}

\begin{proof}
    Since the boundary $B$ is periodic there is a collar neighborhood of $B=A\times S^1$ where $A$ is an annulus containing one periodic orbit $\gamma$ of the supported vector field and such that $\ker \alpha_-$ is tangent $A$ along $\gamma$ and the $S^1$-direction is realized by trajectories of the Reeb vector field of $\alpha_-$.  Since $\gamma$ is Legendrian we can choose a small enough $\tau$ such that the Reeb vector field $R_{\alpha_-}$ is transverse to $A$.  

\begin{figure}[ht!]
    \centering
    \includegraphics[width=0.7\textwidth]{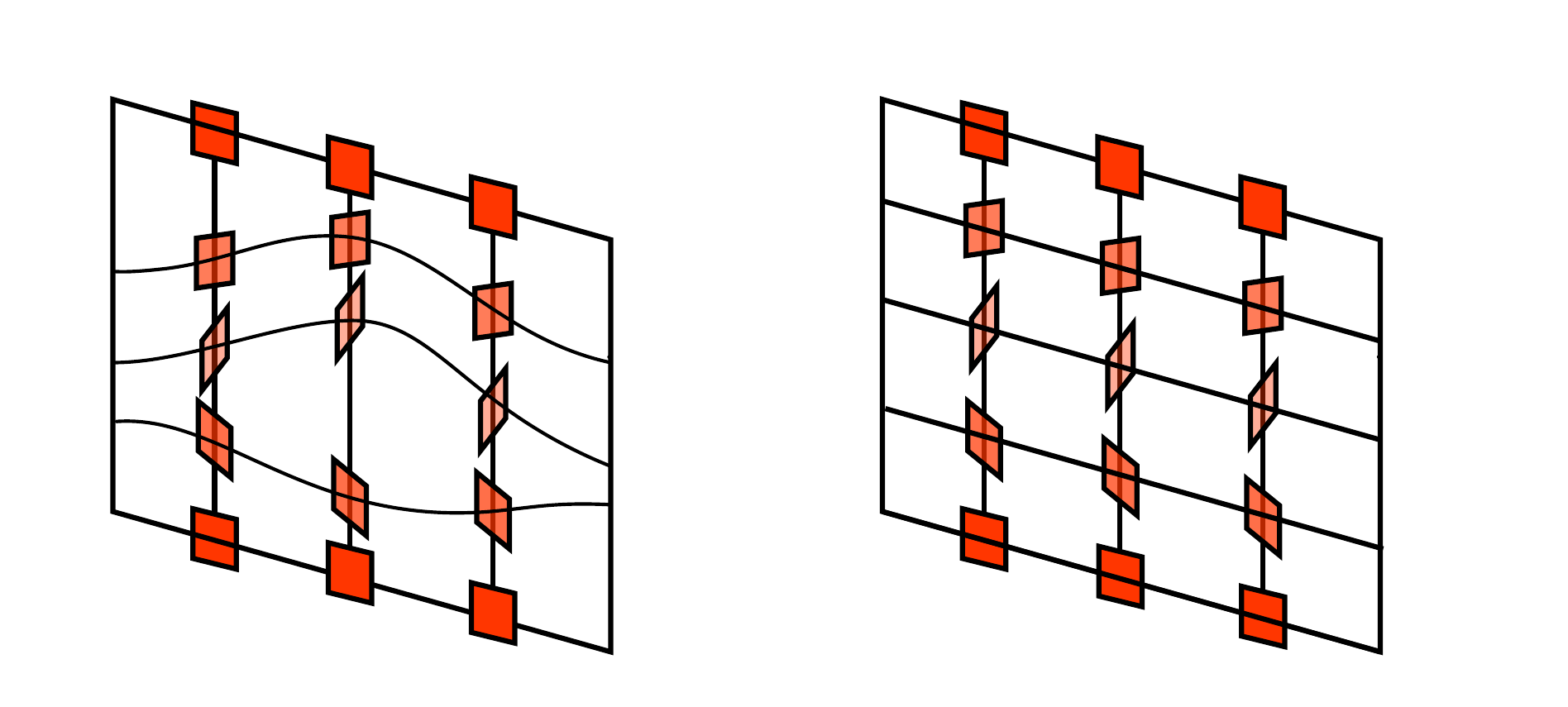}
    \caption{Isotopy of $\xi_+$ along the flow-lines of $R_-$. The black lines on the left are level curves of $\psi(s,v,w)$. }
    \label{fig:Fiso}
\end{figure}

Let $\phi^w$ be the flow of the Reeb vector field $R_{\alpha_-}$ of $\alpha_-$. There is an auxiliary coordinate system $(s,u,w)$ in a neighborhood of $A$ such that 
$$\alpha_-=dw+a(s,u)\:ds.$$
Since $\gamma$ is a Legendrian knot, we have $a(s,0)=0$ and since $\frac{\partial}{\partial u}a(s,u)\neq 0$, the transformation of coordinates $$(s,u)\rightarrow (s,a(s,u))=:(s,v)$$
is non singular and therefore we have (see \cite{FoHa1})
$$\alpha_-=dw+v\:ds.$$

    Since the periodic foliation is Legendrian for $\xi_-$ the contact structure $\xi_+$ monotonically rotates clockwise with $w$ (\cite{ElTh}). Since there are $2n$ boundary orbits, there are $2n$ curves where $\xi_+$ is tangent to $B$, we can write $\xi_+=\ker \alpha'_+$
    $$\alpha'_+=
\sin(n \psi(s,v,w))\:ds+\cos(n \psi(s,v,w))\:dv$$
in a subset of $B$ where $v\in [-\delta,0]$

By the contact condition $\psi(s,v,w)$ is increasing with the coordinate $w$ and we can apply the same argument of (see \cite{Sal1} Lemma 5.5 and Figure \ref{fig:Fiso}) to get
\[
\alpha_+=
\sin(n w)\:ds+\cos(n w)\:dv.\]

\end{proof}

\begin{proof}[Proof of Theorem \ref{thm: gluing plugs}]
Let $M$ be be the closed manifold obtained by gluing $M_1$ and $M_2$ along some compatible boundary component $B_1$ and $B_2$ using the coordinate systems $(w_i,s_i,v_i)$ in $N(B_i)$, $i=1,2$,
we glue $B_1$ to $B_2$ via the map $F$ that
acts by:
\[
\begin{array}{l}
      w_1\to w_2 +\pi\\
     s_1\to-s_2.
\end{array}
\]

Note that when approaching the boundary in the $v_1$ direction, this is glued to the direction $-v_2$ from the boundary inwards in $M_2$. The orientation of $(w_2,s_2,v_2)$ is equal that of $(w_2,-s_2,-v_2)$, and thus there is a global orientation on $M_1\cup_F M_2$ so that the positive (resp.  negative) structures are still positive (negative) as before.

\begin{figure}[ht!]
    \centering
    \includegraphics[width=0.6\textwidth]{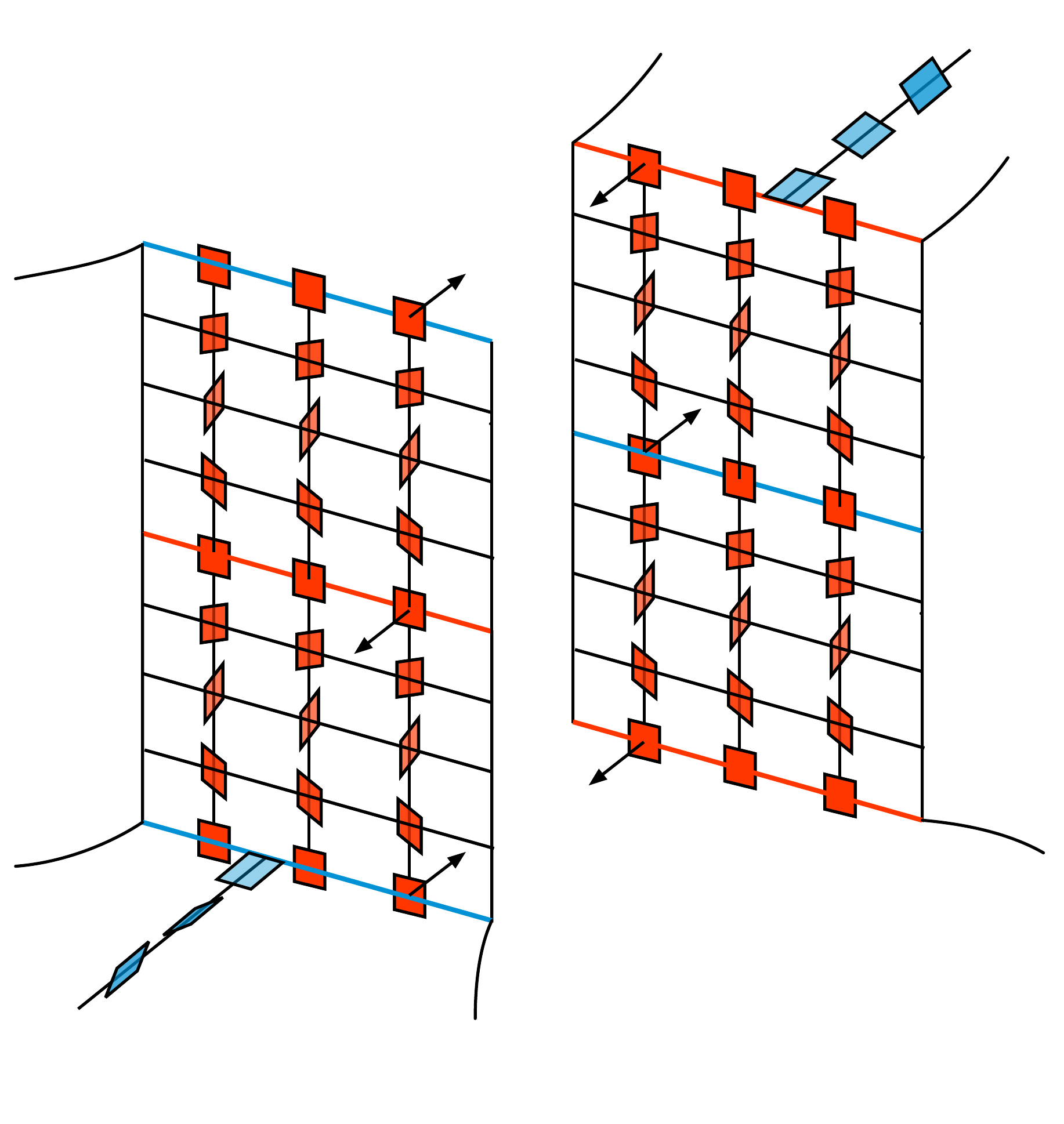}
    \caption{Gluing of two copies of a bicontact plug with quasi-transverse periodic boundary $B$. The gluing map $F:B_1\rightarrow B_2$ identifies points with vertical coordinate $w$ with points $w+\frac{\pi}{n}$. Here red curves are glued to red curves and blue curves to blue curves.}
    \label{fig:Gluing}
\end{figure}

The gluing map sends the 1-form $\alpha_-$ to the 1-form $\alpha_-'$, as gluing the $w_1$-curves with $w_2$-curves corresponds to gluing the flow-lines of the Reeb vector field $R^1_-$ in $M_1$ to the Reeb vector field $R^2_-$. Since we changed the signs of the $\sin$ and $\cos$, and also of $ds_2$ and $dv_2$, it also takes $\alpha_+$ to $\alpha_+'$.
Thus, $M= M_1\cup M_2$ is equipped with a strongly adapted bicontact structure $(\alpha_-,\alpha_+)$. Since the structure of the contact structures in the rest of the manifolds hasn't changed, $M$ is an $\mathcal{SAB}$-plug. If $M$ has empty boundary, 
the flow described by $\alpha_-=\alpha^1_-\cup\alpha^2_-$ and $\alpha_+=\alpha^1_+\cup\alpha^2_+$ is Anosov by \ref{thm:Hoz} (see Figure \ref{fig:Gluing}). 
\end{proof}

\begin{remark}
    On the level of the dynamical system, the gluing takes a tangent orbit to the antipodal tangent orbit (reversing its direction) and the outwards transverse annulus above this orbit to an inwards transverse annulus above its image orbit.
\end{remark}

\section{Strongly adapted bicontact structures on surface bundles over the circle }\label{Sec4}

In this section we show how to construct families of $\mathcal{SAB}$-plugs. We first build from scratch examples of $\mathcal{SAB}$-plugs with ambient manifold $M$ diffeomorphic to a trivial bundle $\Sigma_{g,b}\times S^1$. 
We define on $M$ a bicontact structure $\mathcal{B}^k=(\alpha_-^k,\xi_+^k)$ for any $k\in\mathbb{N}$, such that $\xi_+^n$ is isotopic to $\xi_+^m $ if and only if $n=m$.

\vskip .5 cm

Let $\Sigma_{g}$ be an orientable genus $g$ closed surface, for $g\geq 1$. Consider on $\Sigma_{g}$ a vector field $V$ with $b$ isolated singular points $v_i$ of index $\ind_{v_i}\leq 0$. By the Poincaré--Hopf theorem we have 
   
   $$\sum^b_{i=1} \ind_{v_i}(V)=2-2g.$$
   
   For every $1\leq i \leq b$ consider a disk $D_i$ containing $v_i$ and let $\Sigma_{g,b}= \Sigma_g \setminus \{D_1,\cdots,D_b\}$ and consider the restriction of the vector field $V$ to $\Sigma_{g,b}$. Represent $M=\Sigma_{g,b}\times S^1$ as a mapping torus
   $$M=(\Sigma_{g,b}\times I)/\sim_{Id},$$
   where $I=[0,2\pi]$ and $\sim_{Id}$ identifies a point $(x,2\pi)\in \Sigma_{g,b}\times\{2\pi\}$ with the point $(x,0)\in \Sigma_{g,b}\times\{0\}$. 
   
   On $M$ we can define a vertically rotating positive contact structure as follows: Consider a metric $g(\cdot,\cdot)$ on $\Sigma_{g,b}$ and let $W$ be a non vanishing vector field such that $W$ has an angle of $-\pi/2$ from $V$. Every point $p$ in the interior of $\Sigma_{g,b}$ is contained in a rectangular chart with coordinate system $(x,y)$ induced by the flow-lines of $V$ and $W$. Near a component of $\partial\Sigma_{g,b}$ we consider annular charts $A_i$ with coordinates $(r,\theta)_i$, where $\theta$ is a parametrization of the boundary component, and the $r$-curves are segments transverse to the boundary measuring the distance from the boundary components of $A_i$. 

   The charts on $\Sigma_{g,b}$ can be extended to charts on $\Sigma_{g,b}\times I$ with coordinates $(x,y,w)$ and $(\theta,r,w)$, where $w$ is the parameter along $I$ and the other coordinates define directions tangent to the surface. For any positive integer $k$, we can now construct a positively oriented contact form on $M$ defining it locally in each chart in the interior of $\Sigma_{g,b}$ as     
$$\alpha^k_+=\sin( k w)\:dx+\cos( k w)\:dy.$$

Near the boundaries the situation is as follows. Consider the section at $w=0$ along the $i^{th}$ boundary component of $\partial\Sigma_{g,b}$. Recall that the index $\ind_{v_i}$ of the vector field $V$ is non-positive. As one travels along the boundary counterclockwise, the vector field rotates $1-\ind_{v_i}$ times counterclockwise relative to the tangent direction. The same is true for $W$ as it rotates together with $V$ so as to never intersect it.

Thus, choosing the parametrization of the boundary component so that $\theta=0$ at a point where the $\theta$ curve is tangent to $W$, and denoting $h_i=\ind_{v_i}-1<0$,  we have at $w=0$:
\[
\alpha^k_+=dy=\sin(-h_i\theta)\:d\theta+\cos (-h_i\theta)\:dr
\]

Now, adding clockwise rotation along the vertical direction as $w$ increases, we have in a collar neighborhood of the $i^{th}$ boundary component:
\[
\alpha^k_+=\sin (k w-h_i\theta)\:d\theta+\cos(k w-h_i\theta)\:dr,\;\;\;w,\theta\in[0,2\pi].\]

Here $k\in \mathbb{N}$ is the rotation number of the plane field $\ker{\alpha_+}$ along the vertical direction.

\begin{figure}
    \centering
    \includegraphics[width=1\textwidth]{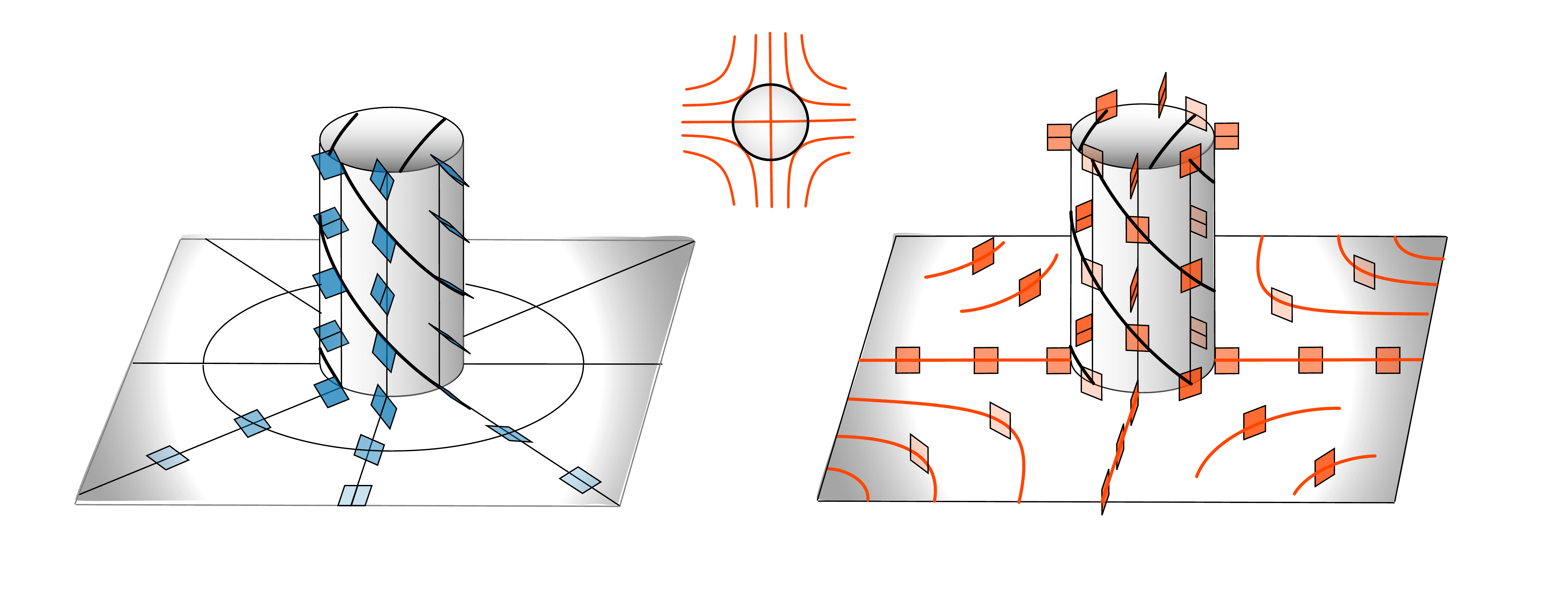}
    \caption{Bicontact structure near a boundary component of index $-1$. On the left the vertically invariant negative contact structure $\xi_-=\ker \alpha_-$. On the right, the vertically rotating positive contact structure. The horizontal surface is $\Sigma^0_{g,b}$ and the orange curves are flow-lines of $V$.}
    \label{fig:Singular.pdf}
\end{figure}

By construction, the trajectories of $V$ can be interpreted as the leaves of the characteristic foliation induced by $\xi_+=\ker \alpha_+$ on $\Sigma^0_{g,b}=\Sigma_{g,b}\times \{0\}=\Sigma^{2\pi}_{g,b}=\Sigma_{g,b}\times \{2\pi\}$.

\begin{proposition}\label{prop: bicontact plugs}
  For any natural number $k$ and a vector field $V$ with non-positive singularities on a closed surface  $\Sigma_g=\Sigma_{g,b}\bigcup_{i\in [1,\cdots,b]} D_i$, there exists a bicontact structure $\mathcal{B}^k_V$ on $M=(\Sigma_{g,b}\times I)/\sim_{Id}$, that is a $\mathcal{SAB}$-plug with $\mathcal{QTP}$-boundary. 
\end{proposition}

 \begin{proof}
    Following Thurston-Winkelnkemper \cite{ThurstonWinkelnkemper1975}, we construct a contact form $\alpha^k_-$ defining a negative contact structure that is vertically invariant, i.e., such that $R_-=\frac{\partial}{\partial w}$.
 
 Near a point $p$ of the interior of $\Sigma_{g,b}$ chose
 $$\alpha^k_-=dw+\beta,$$ where $\beta$ is a 1-form on $\Sigma_{g,b}\times \{w\}$ that does not depend on $w$ and such that $d\beta$ is an area form. In a collar neighborhood $N_i$ of a connected component $B_i$ of $\partial M$ we may choose $\beta=(r-h_i/k)\:d\theta$ and define 
$$\alpha^k_-=dw+(r-h_i/k)\:d\theta,\;\;\;r\in[-\delta,0].$$

Note that since $\alpha^k_+(\frac{\partial}{\partial w})=0$, i.e. the vertical direction is always contained in the plane field $\ker \alpha^k_+$ as defined above, the pair of contact forms $(\alpha^k_-,\alpha^k_+)$ defines a strongly adapted bicontact structure.

When we choose the same $k$ for both structures, the tangency lines of $\alpha^k_+$ on $B_i$ are contained in $\alpha^k_-$. Thus, the boundary component $B_i$ is quasi-transverse and periodic. 
 \end{proof}

Note that the bicontact structures constructed in Proposition \ref{prop: bicontact plugs} satisfy the following properties
\begin{enumerate}
    \item $\xi_+^n$ is isotopic to $\xi_+^m$ if and only if $n = m$;
    \item the Reeb vector field $R_-^k$ of $\alpha_-^k$ is periodic and transverse to $\Sigma_{g,b}^w := \Sigma_{g,b} \times \{w\}$ for every $w \in [0, 2\pi]$;\footnote{This condition ensures that $\Sigma_{g,b}^w$ is convex in the sense that there exists a contact vector field $v = R_-$ transverse to $\Sigma_{g,b}^w$.}
    \item the contact structure $\xi_-^k$ is transverse to $\partial \Sigma_{g,b}^w$ for every $w \in [0, 2\pi]$ and can be isotoped to be arbitrarily close to $\Sigma_{g,b}^w$ away from $\partial \Sigma_{g,b}^w$;
    \item the characteristic foliation induced by $\xi_+^k$ on $\Sigma_{g,b}^0$ is directed by $V$;
    \item each boundary component is quasi-transverse with $2n$ boundary orbits of $X^t$ and foliated by closed orbits of $R_-^k$.
\end{enumerate}

    \begin{remark}
        Note that on $B_i$ we have $r=0$. Therefore $\alpha^k_-|_{\partial N_i}=dw-h_i/k\:d\theta$. The characteristic foliation induced by $\ker \alpha^k_-$ on $B_i$ is directed by $u=\frac{\partial}{\partial \theta}+\frac{h_i}{k}\frac{\partial}{\partial w}$, therefore its leaves are lines with negative slope, parallel to lines satisfying the equation $w=\frac{h_i}{k}\:\theta$.
    \end{remark}

\begin{remark}\label{rem:orbit slopes}
Let $\gamma$ be a closed orbit on $B_i$ such that $[\gamma]=p[w]+q[\theta]$ where $[w]$ represents the homology class of the (periodic) orbits of $R_-$ while $[\theta]$ represents the class of the $i^{th}$-boundary component of $\Sigma_{g,b}$. Note that $\frac{p}{q}=\frac{h_i}{k}$ is the slope of $\gamma$ and we have $p=\frac{h_i}{\gcd(k,h_i)}$ and $q=\frac{k}{\gcd(k,h_i)}$. Moreover, it follows that $B_i$ contains exactly $2\gcd(k,h_i)$ orbits of the flow defined by $(\alpha_-^k,\alpha_+^k)$.

Thus, the $w$-curves intersect $\gamma$ once on $B_i$ if and only if $q=\frac{k}{\gcd(k,h_i)}=1$, therefore if and only if $k|h_i$. In this case we have $2\gcd(k,h_i)=2k$ closed orbits of $X^t$ on $B_i$. 
\end{remark}

   Given the quasi-transverse periodic torus $B_i$, proposition \ref{pro:one intersection} ensures that there is a Legendrian-transverse surgery along a L-t push-off of a closed orbit $\gamma$ such that the new $B_i$ is a quasi-transverse periodic torus such that the trajectories of the new Reeb vector field $\tilde{R}_-$ intersect $\gamma$ at a single point.   

\vskip .5cm
Let $[\gamma]=p[w]+q[\theta]$ with $\gamma$ a closed orbit of $X^t$. A curve $[c]=r[w]+s[\theta]$ intersects $\gamma$ once if and only if $|qr-ps|=1$.
 Call $\tilde{w}$-curves the trajectories of $\tilde{R}_-$. The $w$-curves and the $\tilde{w}$-curves coincide far from the boundary while on the boundary we have $[\tilde{w}]=r[w]+s[\theta]$. Here $s$ is as in the proof proposition \ref{pro:one intersection} and $r$ can be chosen to satisfy the relation $qr-ps=1$.

   \begin{remark}
    The new manifold $\tilde{M}$ obtained by performing a surgery as in proposition 4.5 is diffeomorphic to $M$.
Call $w$-curves the trajectories of $\tilde{R}_-$. The $w$-curves and the $\tilde{w}$-curves coincide far from the boundary while on the boundary instead we have $[\tilde{w}]=r[w]+s[\theta]$. In the coordinate system $(w,\theta)$ the boundaries of the fibers of the surface fibration satisfy $\tilde{\Sigma}^w_{g,b}=\tilde{\Sigma}_{g,b}\times \{w\}$, and coincide with the $\theta$-curves. 
\end{remark}

\begin{remark}\label{rmk:relation to geodesic flows}
Suppose we choose a periodic or trivial monodromy, so that our process results in an Anosov flow on a Seifert fibered manifold. Theorem~\ref{thm: gluing plugs} proves this piece can be glued to itself, producing an Anosov flow on a closed manifold that is either Seifert fibered or a graph manifold. In both cases it follows from the work of Barbot \cite{Barbot96} and Barbot Fenley \cite{BaFe4} that each piece must be a cover of a geodesic flow (since the same piece is glued to itself and the boundary orbits are not homotopic to a fiber, the flow must consist of so called `free pieces without blow-ups').
On the other hand, let $\Sigma_{g,b}$ be a genus $g\geq 1$ surface equipped with a metric of negative curvature and $b$ geodesic boundary components. The geodesic flow on the unit tangent bundle $T^1\Sigma_{g,b}$ is defined by a bicontact structure described in \cite{FHV2}, that seemingly coincides with the structure $\mathcal{B}^1$ on the $\mathcal{SAB}$-plug $\mathcal{M}^1_V$ for some vector field $V$ on $\Sigma_{g,b}$. 
Cutting along a close curve tangent to a periodic orbit of a section for a torus implies the construction yields also the geodesic flow for a 3-punctured sphere.
The degree $k$ cover is thus given by $\mathcal{M}^k_V$. Thus it seems that any cover of a geodesic flow will arise from our construction, and vice versa, all Anosov flows arising from our plugs in Seifert fiber spaces are covers of geodesic flows.
\end{remark}

\section{Modifying the monodromy by bicontact surgery}

In this section we show how to modify the monodromy of the mapping torus associated to $\mathcal{M}^k_V$ by bicontact surgery along a special family $\Gamma$ of closed orbits of $X^t$. The key step is to realize a class of curves that belong to a fiber $\Sigma^w_{g,b}$ simultaneously as Legendrian curves for $\ker \alpha_-$ and $\xi_+$. Curves of this type are called biLegendrian and are exactly the closed orbit of the flow supported by the bicontact structure, allowing us to preform surgeries along them. We must first introduce some background in convex surfaces theory introduced by Giroux (see \cite{Gir01}).

\begin{definition}
    A vector field $v$ in a contact $3$-manifold $(M,\xi)$ is \textbf{contact} if $v$ preserves $\xi$.
\end{definition}

\begin{remark}
    The Reeb vector field $R_-$ of $\alpha_-$   is a special contact vector field: it preserves not just the contact structure $\ker \alpha_-$ but also the contact form $\alpha_-$. Since it is Legendrian for $\xi_+$ it nowhere preserves $\xi_+$.
\end{remark}

\begin{definition}
    An embedded closed surface $\Sigma$ in a contact manifold $(M,\xi)$ is convex \cite{Gir01} if there is a contact vector field $v$ transverse to $\Sigma$.
\end{definition}
In other words there is a neighborhood $N=\Sigma\times I$ such that the contact structure $\xi$ is invariant in the $I$ direction. The behaviour of the contact structure in $N$ is described by a surprisingly simple piece of information: 

\begin{definition}
The dividing set is the multi-curve  $\Gamma_\Sigma=\{x\in \Sigma \:|\:v(x)\in \xi\} $, i.e. the set of curves where the contact vector field is contained in the contact structure. 
\end{definition}

\begin{definition} 
\label{defn:non isolating curve}
An embedded graph $G$ on a convex surface $\Sigma$ is \textbf{non-isolating} if every connected component of $\Sigma\setminus c$ intersects $\Gamma_{\Sigma}$.

\end{definition}

\begin{theorem}[Legendrian Realization Principle \cite{Et04}]
\label{LRP}
Let $\Sigma$ be an embedded convex surface in a contact manifold $(M,\xi)$. Let $c$ be an embedded closed curve that is non-isolating. Then there is an isotopy $\psi_s \colon \Sigma \rightarrow M,\; s\in [0,1]$ such that
\begin{itemize}
    \item $\psi_0=id,\; \psi_s|\Gamma_{\Sigma}=id$ ($\psi_s$ is fixed on the dividing set),
    \item $V$ is transverse to $\psi_s(\Sigma)$ (therefore $\psi_s(\Sigma)$ is convex for all $s\in [0,1]$),
    \item $\psi_1(\Gamma_{\Sigma})=\Gamma_{\psi_1(\Sigma)}$
    \item $\psi_1(c)$\; is\; a Legendrian simple closed curve.

\end{itemize}
\end{theorem}

The Legendrian Realization Principle in the above form applies to closed embedded convex surfaces. However, its use can be extended to our context. Let $\Sigma_{g,b}$ be a surface with boundary and consider  $M=\Sigma_{g,b}\times S^1$, the ambient manifold of $\mathcal{M}_V^k$, equipped with a contact form $\alpha^k_-$ with periodic Reeb vector field $R_-$ as constructed in Section \ref{Sec4}. Consider the embedded surface $\Sigma^w_{g,b}=\Sigma_{g,b}\times \{w\}$. $\Sigma^w_{g,b}$ is transverse to the contact vector field $R_-$. Glue to each boundary component of $M$ an appropriate solid torus $D^2\times S^1$ equipped with a radially symmetric contact form that is invariant in the $S^1$ direction. The resulting manifold $M'=\Sigma_g\times S^1$ is equipped with a contact form $\alpha'_-$ that coincide with $\alpha_-$ on $M\subset M'$. For every $w\in [0,2\pi]$ the surface $\Sigma^w_g$ is convex in $(M',\ker \alpha_-')$ with dividing set given by a disjoint union of contractible\footnote{Therefore by Giroux criterion  $\ker \alpha_-'$ is {\it overtwisted}.} curves isotopic to the boundary components of $\Sigma^w_{g,b}$. Since a non separating curve $c$ on $\Sigma^w_{g,b}$ is non-isolating in $\Sigma^w_g$ we have the following:

\begin{remark}
    A simple closed curve on $\Sigma^w_{g,b}$ that is non-separating is non-isolating. 
\end{remark}

\begin{corollary}
Let $c$ be a simple closed curve on $\Sigma^w_{g,b}$ that is non-separating. After an isotopy of $\Sigma^w_{g,b}$ transverse to $R_-$ we can assume that $g$ is Legendrian for $\ker \alpha_-$.
\end{corollary}

\begin{figure}[ht!]
    \centering
    \includegraphics[width=0.85\textwidth]{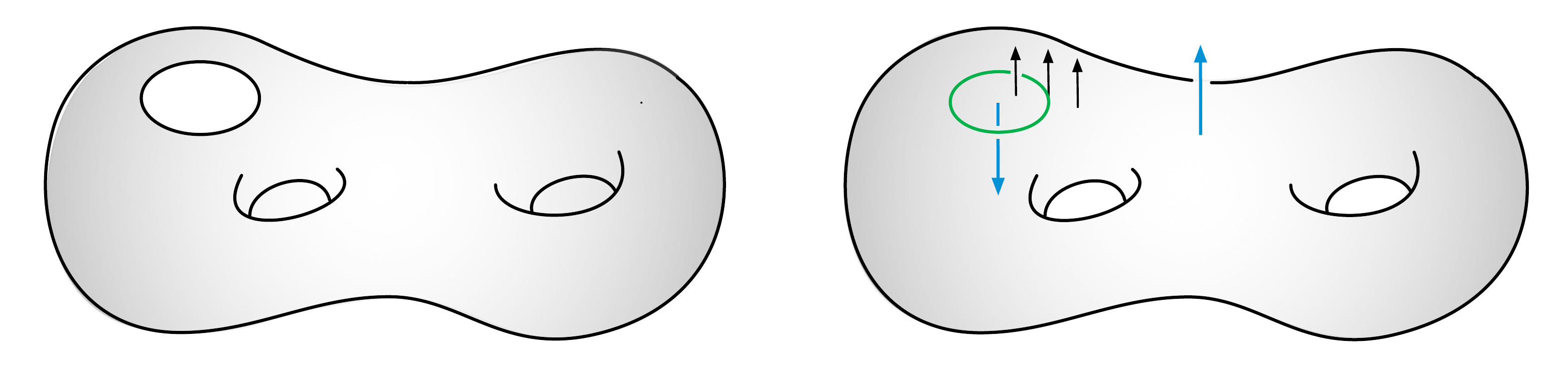}
    \caption{The surface with boundary $\Sigma_{g,b}$ is transverse to $R$. When we cap off the boundary components of  $\Sigma_{g,b}\times S^1$ we get a convex surface with dividing set isotopic to $\partial\Sigma_{g,b}$. The black arrows represent the contact vector field $v$, the blue ones the Reeb vector field. The green curve is the dividing set.}
    \label{fig:Convex}
\end{figure}

\begin{remark}
    Note that the Reeb vector field $R$ of a contact form such that $\ker \alpha=\xi$ is a contact vector field that is not everywhere transverse to a closed surface. In the case of our capped-off surface $\Sigma_g$ we can ensure that the contact vector field $v$ coincide with $R$ on $\Sigma_{g,b}$ and coincide with the opposite of the Reeb field $R_D$ of the radially symmetric contact form only in a neighborhood $N$ of the core of the solid torus that we used to fill the boundary. Moreover $R$ points outwards $\Sigma_g$ on $\Sigma_{g,b}$ while points inwards $\Sigma_g$ on $N$ (see Figure \ref{fig:Convex}).   
\end{remark}

We now focus on the second step: realizing $c$ as a closed leaf of the characteristic foliation induced by $\xi_+$ on $\Sigma^w_{g,b}$.

\begin{definition}
 Let $s\rightarrow c(s), s\in S^1=[0,2\pi)$ be a parametrization of $c$ and let $V$ be the vector field directing the characteristic foliation on $\Sigma^0_{g,b}$. We define the \textbf{winding function} $wind_V(c):[0,2\pi)\rightarrow \mathbb{Z}$ as the function that counts the points where $Tc\in \xi_+$ following the parametrization of $s\rightarrow c(s), s\in [0,2\pi)$ with the convention that for a neighborhood $N(p)$ of $p$ the value of a tangency point $c(s_p)=p$ along the parametrization is
 \begin{itemize}
     \item $+1$ if the flow-line of $V$ containing $p$ is on the left side of $c$ in $N(p)\setminus p$
     \item $-1$ when the flow-line of $V$ containing $p$ is on the right side of $c$ in $N(p)\setminus p$ \item $0$ if the flow-line of $V$ is contained in both the sides of $c$ in $N(p)$.
 \end{itemize}
   
\end{definition}

\begin{definition}
    The \emph{cardinality} of the image of the winding function of $c$ with respect to $V$ in $\mathbb{Z}$ is the difference between the maximum and minimum value of the winding function $wind_V(c)(s)$. We denote the cardinality by $\Delta W_V(c)$ 
\end{definition}

\begin{remark}
While the winding number $wind_V(c)$ is invariant up to isotopy of $c$ the winding function depends on the specific curve $c$ and on the chosen parametrization.
However, the quantity $\Delta W_V(c)$ does depend only on the curve $c$ and it is independent on the parametrization. 
\end{remark}

\begin{proposition}[Realizing biLegendrian curves]
\label{prop:wind}
      Let $\mathcal{M}_V^k$ be a trivial strongly adapted bicontact plug. Let $V$ be the vector field directing the  $\xi_+$ on $\Sigma^0_{g,b}$ and let $c$ be a non separating simple closed curve on $\Sigma^0_{g,b}$ with winding number $wind_V(c)=0$. Then for $k$ large enough there is a closed orbit $\gamma$ of $X^t$ isotopic to $c$.
    \end{proposition}

\begin{proof}
Since $c$ is non-separating after an isotopy of $\Sigma^0_{g,b}$ by Theorem \ref{LRP} there is a curve $c'$ that is Legendrian for $\ker \alpha_-$ and isotopic to $c$. 
Since the winding number is an invariant of isotopy we have $wind_V(c')=0$. Consider the cardinality of the image of the winding function, namely $\Delta W_V(c')$. Suppose that $k\geq 2\Delta W_V(c')$. Consider the open cylinder $\mathcal{C}=c'\times (0,2\pi)$ embedded in $\mathcal{M}_V^k$. Since $k\geq 2\Delta W_V(c')$ there is a simple closed curve $c''$ on $\mathcal{C}$ and isotopic to $c'$ where $\xi_+=T\mathcal{C}$ or equivalently $c''$ that is a Legendrian curve in $\mathcal{C}$ for $\xi_+$ (see Figure \ref{fig:Normal}). Now, the curve $c'$ and $c''$ are both contained in $\mathcal{C}$ and both intersect each segment of the vertical fibration $C$ once. Since the contact structure $\xi_+$ contains the vertical fibration of $\mathcal{C}$ we can isotope the contact structure $\xi_+$ along the vertical fibration in such a way that after the isotopy $c''=c'$ (see \cite{Sal1} Lemma 5.5 and Figure~\ref{fig:Normal}). Note that such an isotopy preserves transversality since it is achieved by sliding $\xi_+$ along curves that are transverse to $\ker \alpha_-$. Then, after isotopy $c''=c'=\gamma$ is a closed orbit of the projectively Anosov flow supported by the bicontact structure $(\ker \alpha_-,\xi'_+)$ where $\xi_+'$ is isotopic to $\xi_+$. We conclude the proof of the statement noting that we can think of the plug as embedded in an Anosov flow. This follows as along the isotopy the strongly adapted property is preserved, and so the isotopy induces a one parameter family of flows. At each stage of the isotopy two copies of the manifold $M$ are compatible $\mathcal{SAB}$ plugs with $\mathcal{QTP}$ boundaries, and so the ``doubled" flow is supported by an $\mathcal{SAB}$ structure on the double, and thus is Anosov. The structural stability implies these flows are all equivalent.

\end{proof}

\begin{figure}
    \centering
    \includegraphics[width=0.65\textwidth]{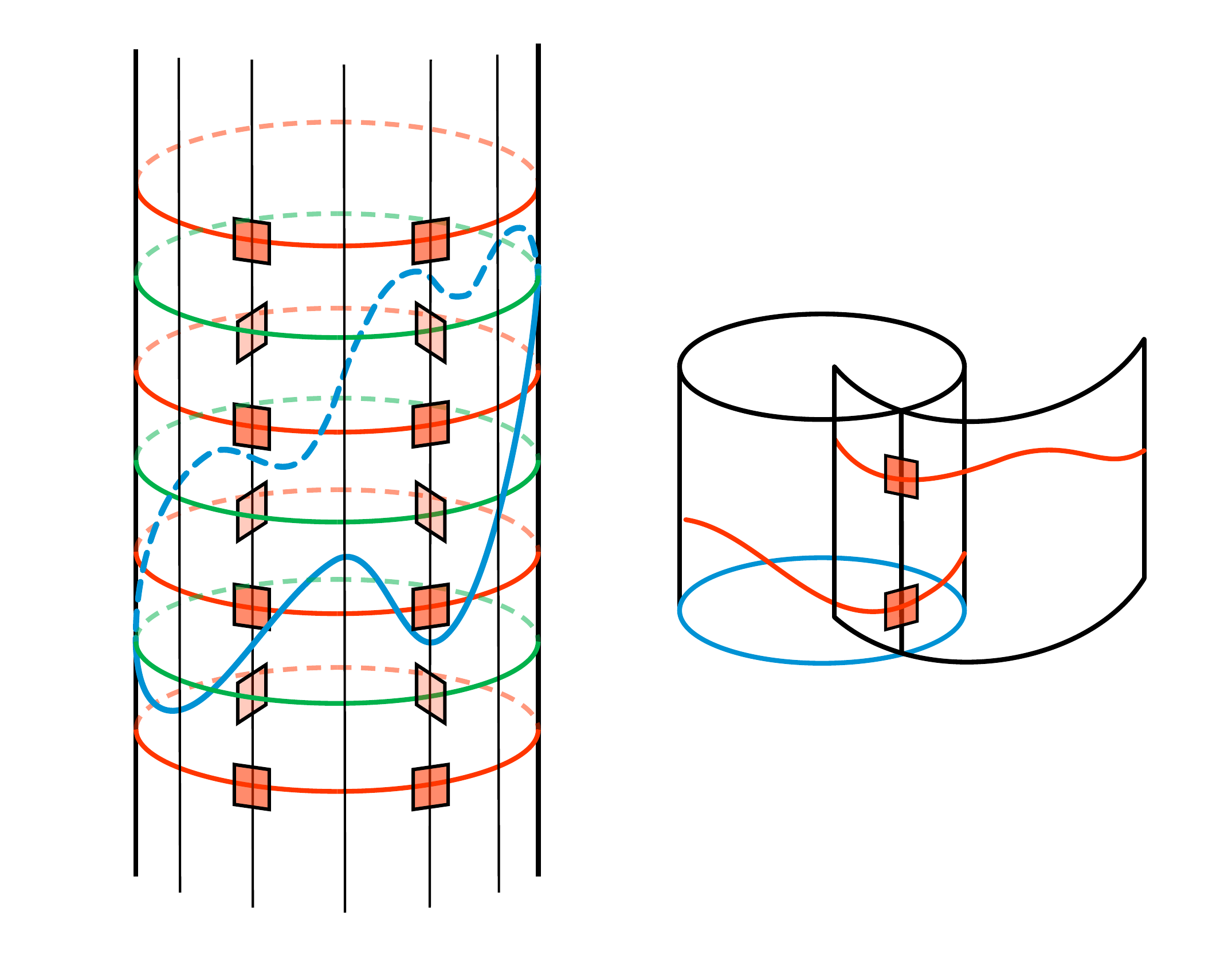}
    \caption{The cylinder $\mathcal{C}$ containing $c'$. The blue curve $c'$ is such that  $wind_V(c')=0$ and $\Delta W_V(c')=1$. A red curve $c''$ is isotopic to $c'$. On the right, two cylinders above two intersecting curves $c_1$ and $c_2$.  }
    \label{fig:Normal}
\end{figure}

We can now proof our main result.

\begin{figure}[!ht]
    \centering
    \includegraphics[width=0.8\textwidth]{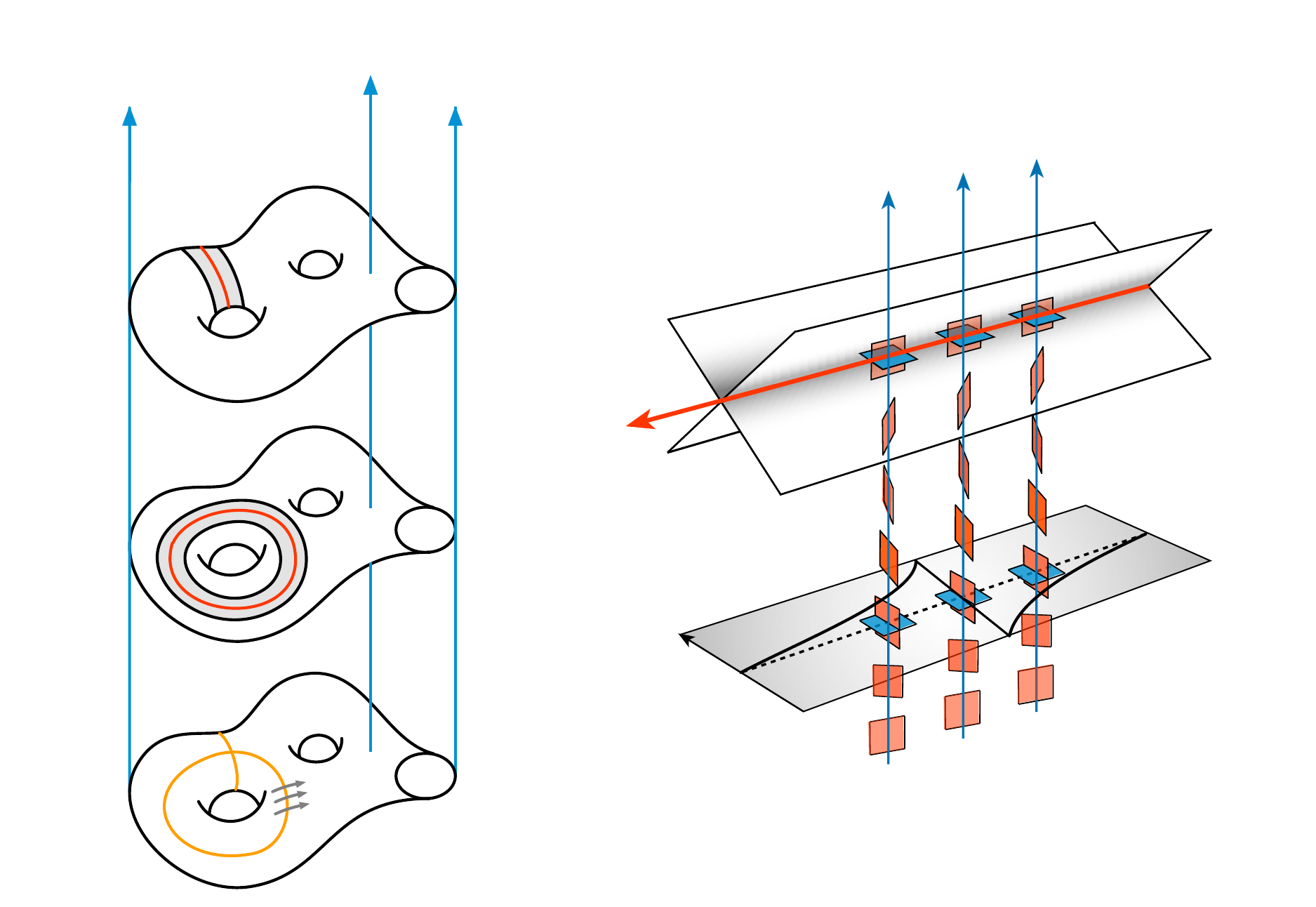}
    \caption{The yellow graph $G=c_1\cup c_2$ is the union of curves that are transverse to a vector field $V$. The red curves $\gamma_i$ are biLegendrian at level $w_i$ and parallel to $\Sigma^{w_i}_{g,b}$ . The surgery annuli $A_i$ have the same framing of the fibers by construction. A bicontact surgery along a bilegendrian knot $\gamma_i$ changes the monodromy by a Dehn twist along $\gamma_i$.}
    \label{fig:Surgery}
\end{figure}

\begin{proof}[Proof of Thorem~\ref{thm:bundle plugs}]
For every $c_i\in \Gamma$ choose a surface $\Sigma_{g,b}\times \{w_i\}$ with $w_1<w_2<\cdots <w_r$. Let $l_i$ be the segment of $w$-curve described by $w\in [w_i,w_{i+1}]$. Moreover assume that $w_{i+1}-w_i=2\pi/r$. First realize $c_i$ as a Legendrian curve $c'_i$ isotopic to $c_i$ on $\Sigma_{g,b}\times \{w_i\}$. Since $c$ is transverse,  $wind_v(c'_i)=0$ by Proposition \ref{prop:wind} we can isotope $\xi_+$ along the flow-lines of $R_-$ to realize $c'_i$ as a biLegendrian curve on $\Sigma_{g,b}\times \{w_i\}$ and therefore as closed orbit of $X^t$. If we take $$k=2\sum_1^r\Delta W_V(c'_i)$$ by the same argument used before we can realize a sequence $\{\gamma_1,\cdots ,\gamma_r\}$ of biLegendrian curves with each curve $\gamma_i$ of the sequence lying on $\Sigma_{g,b}\times \{w_i\}$. Consider a $f=\tau^{q_1}_1\circ \cdots \circ \tau^{q_r}_r\in\MCG(\Sigma_{g,b})$. If we perform a $q_i$-bicontact surgery along $\gamma_i$ we change the monodromy from the identity $Id$ to $\tau_i^{q_i}$ (in symbols  $Id\xrightarrow {\mathcal{S}\{(\gamma_i,q_i)\}}\tau_i^{q_i}$). Note that a composition of a $q_i$-bicontact surgery and a $q_j$-bicontact surgery induce a change of monodromy by $\tau_i^{q_i}\circ\tau_j^{q_j}$ (in symbols  $Id\xrightarrow {\mathcal{S}\{(\gamma_{i},q_{i}),(\gamma_j,q_j)\}}\tau_{i}^{q_{i}}\circ\tau_j^{q_j})$ if and only if $w_i<w_j$ (see Figure \ref{fig:Surgery}). In particular the order of the composition is determined by the level of the section $\Sigma_{g,b}^w$ and not by the order of the sequence of surgeries. Inductively we have 

$$Id\xrightarrow {\mathcal{S}\{(\gamma_{1},q_1), \cdots ,(\gamma_r,q_r)\}}\tau_1^{q_1}\circ \cdots \circ\tau_r^{q_r}$$ 
\end{proof}

\begin{remark}
    Theorem \ref{thm:bundle plugs} holds more generally for winding number-0 curves in $V$.
\end{remark}

\section{Punctured torus bundles over the circle and manifolds that support many Anosov flows}

Given a surface $\Sigma_{g,q}$ and a mapping class $f\in \PMod(\Sigma_{g,q})$ it is not in general possible to find a non vanishing vector field $V$ such that $f$ can be decomposed as a product of Dehn twists along a family of curves $C=\{c_1,c_2,\cdots,c_r\}$ such that each $c_i$ is a $wind_V(c_i)=0$. However, in the case of a genus one surface with $b$ boundary components $T^2_b$, this is always possible.

\begin{proposition}
\label{GeneratingV}
    There is a vector field $V$ on $T^2_b$ such that there is a set of $b+1$ curves $C=\{c_1,c_2,\cdots,c_{b+1}\}$ such that transverse to $V$ such that $\PMod(T^2_b)$ is generated by Dehn twists along curves of $C$. 
\end{proposition}

\begin{proof}
Construct on $T^2_b\times S^1$ a contact form $\alpha_-$ as in Section \ref{Sec4}. Consider the curves $C=\{c_1, c_2,\cdots, c_{b+1}\}$ generating $\PMod(T^2_b)$ as in Figure~\ref{fig:Torus}. the  graph $G=c_1\cup c_2\cup \cdots \cup c_{b+1}$ on $T^2_b$ is non-isolating, by the Legendrian Realization Principle \ref{LRP} on $T^2_b$ there is a Legendrian graph $G'$ isotopic to $G$.

    Consider on $T^2=T^2_b\cup D_1\cup\cdots \cup D_b$ a linear foliation $\mathcal{F}$ constructed as follows. Let $c_{b+1}\subset G$ be as above and extend it to a foliation $\mathcal{F}$ with leaves isotopic to $c_{b+1}$ and transverse to $\{c_1,\cdots,c_b\}$ (see Figure \ref{fig:Torus}).
\end{proof}

\begin{figure}[!ht]
    \centering
    \includegraphics[width=0.8\textwidth]{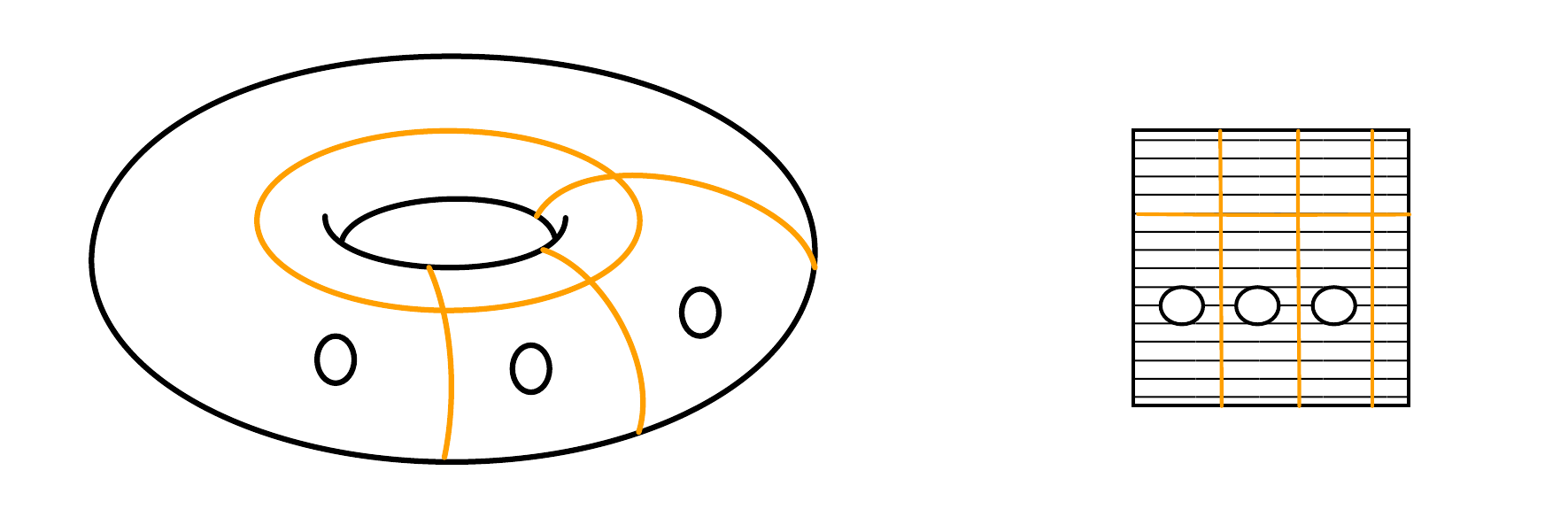}
    \caption{Explicit construction of the vector field $V$ on a $b$-times punctured torus. The orange curves are generators of $\PMod(T^2_b)$.}
    \label{fig:Torus}
\end{figure}

 \begin{proof}[Proof of Theorem \ref{pro:torus}]
     Note that $c_{b+1}$ is a closed trajectory of $V$ by construction while $\{c_1,\cdots,c_b\}$ are transverse to $V$ (note that this implies that $wind_V(c_i)=0, 1\leq i\leq b+1$). For $k=r$ there are exactly $2k$ Legendrians on the cylinder $\mathcal{C}_i$ for every $ 1\leq i\leq b+1$. Each of these curves can be realized as biLegendrian $\gamma_i$ by an isotopy that does not break the strongly adapted property and therefore induces a path of Anosov flows. By theorem and theorem for $k\geq r$ we can realize by surgery along biLegendrians every word $f=\tau^{q_1}_1\circ \cdots \circ \tau^{q_r}_r\in \PMod(T^2_b)$. Consider a boundary component $b_i$ of $T^2_b$. Since by construction $h_i=1$, by remark 5.3 we have $2 gcd(k,h_i)=2$ boundary orbits.
 \end{proof}

\begin{construction}
    Let $T_b^2$ a $b$ times-punctured torus and let $V$ be the vector field constructed in Proposition \ref{GeneratingV}. Consider the strongly adapted bicontact plug with quasi-transverse periodic boundary components $\mathcal{M}$ as constructed in Section \ref{Sec4}. Theorem \ref{thm:bundle plugs} allows us to modifiy its geometry in a controlled way performing bicontact surgery along closet orbits of the flow supported by $(\alpha_-,\xi_+)$ and parallel to $T_b^2$. The result is a new strongly adapted bicontact plug $\mathcal{M}_f$ where $f$ is any monodromy map $f\in \PMod(T_b^2)$. The monodromy is the first return map of the new Reeb vector field $\tilde{R}_-$ of the supporting contact structure $(\tilde{\alpha}_-,\tilde{\xi}_+)$. Since the boundary of the new plug  is quasi-transverse and periodic glue,  we can construct a bicontact plugs $\mathcal{M}_V^k$
\end{construction}

Yang and Yu ask in \cite[Question 1.3]{YangYu2022Classifying} if a toroidal manifold composed of two figure eight knot complement pieces can carry a transitive Anosov flow. We show that in fact it can carry many.

\begin{proof}[Proof of Corollary \ref{cor:figure8}]
Consider the figure-eight knot complement $M_8$ as a once punctured torus bundle over the circle, with the Arnold cat map monodromy. Let $\phi_n$ be the Anosov flow corresponding to the positive contact structure twisting $n$ times along the circle.
The flow $\phi_n$ has two tangent orbits to $\partial M_8$, with coordinates $\mu+n\lambda$ in the homology of the boundary, where $\mu$ is the meridian of the figure-eight knot and $\lambda$ the longitude. 
Note that $\lambda$ intersects the orbit once.
The orbits of the positive Reeb field $R_+$ are parallel to $\mu$, and thus intersect the orbits $n$ times. After the boundary bicontact surgery, the new Reeb orbits $w'$ have homology 
\[
w'=\frac{1}{n}\mu+\frac{s}{n} (\mu+n\lambda),
\]
where $s$ is defined mod $n$. In other words, we are free to sheer along the orbit, connecting the same point on the bottom boundary of an annular neighborhood of the orbit push off to the same point on the top boundary of the neighborhood, obtaining any closed curve that intersects the orbit once as an orbit of the new Reeb flow.

Note that for any $n$, the longitude $\lambda$ is also a closed curve intersecting the orbit once. 
Therefore, we may choose the new Reeb orbits $w'$ to be parallel to the longitude $\lambda$. Make the same choice for any flow $\phi_n$ for any $n\in\mathbb{N}$. The map defined in Theorem~\ref{thm: gluing plugs} gluing the $B_1$ plug to the $B_{2k}$ plug is in homology of the boundary: 

\[
\Psi_k: [\mu]\to -[\mu]-2k[\lambda],\ \Psi:[\lambda]\to [\lambda].
\]

(The translation along lambda doesn't effect the homology).
$\Psi_k$ is chosen so that is glues a tangent orbit $o_n=\mu+n\lambda$ for $\phi_n$ to a tangent orbit $o_{2k-n}=\mu+(2k-n)\lambda$ for $\phi_{2k-n}$, for any $1\leq n\leq k$, and the orbits of the Reeb flow to orbits of Reeb flow, up to sign.
Thus, the resulting manifold $M_k=M_8\cup_{\Psi_k}M_8$ carries $k+1$ different bi-contact structures, and $k+1$ Anosov flows on it. 

Note that when traveling along an orbit of the Reeb flow $R'_+$ on the boundary, viewing the boundary from the inside of $M_8$, one sees that as the plane field rotates right-handedly, once crossing an orbit going from left to right one arrives at an outgoing annulus. This is sent to an orbit going right to left, and continuing in the direction of $R'_+$ one reaches the inwards pointing transverse annulus. Thus, the gluing map matches orbit to orbit, and an outwards transverse annulus to an inwards transverse annulus.

Denote the flow obtained by attaching $\phi_n$ to $\phi_m$ by $\phi_{n,m}$.
We claim the $k+1$ flows thus defined on $K_k$ are not orbit equivalent to each other:
The manifold $M=M_8\cup_{\Psi}M_8$ contains a unique essential torus $T=\partial M_8$. Therefore, any orbit equivalence between two of these flows $\phi_{n,m}$ and $\phi_{n',m'}$ takes $T$ to $T$. Cutting the manifold $\Sigma$ along $T$ we obtain two pieces homeomorphic to the figure eight complement. Therefore, the orbit equivalence restricted to one piece must take it to one of the pieces $\Sigma_{n'}$ or $\Sigma_{m'}$. But as the meridian and longitude are well defined in $M_8$ (up to sign), and thus so are the slopes of the boundary orbits, an orbit equivalence is only possible is $\{n,m\}=\{n',m'\}$ (see Figure \ref{fig:Fibers}).
\end{proof}

\begin{figure}
    \centering
    \includegraphics[width=0.8\textwidth]{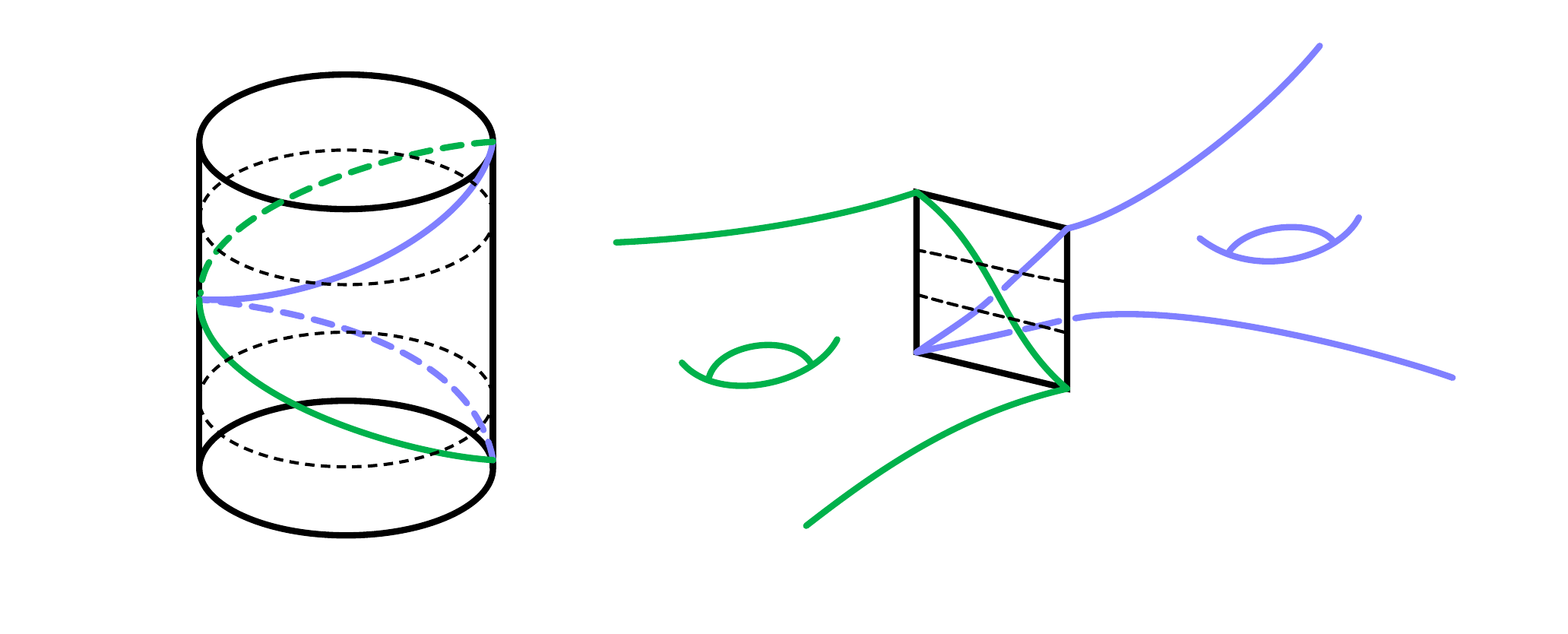}
    \caption{Interaction of the fibers along a quasi-transverse periodic boundary. The dashed black curves are the boundary obits of the supported flow. The vertical direction represents $R_-$.  }
    \label{fig:Fibers}
\end{figure}

    The same construction allows us to give many examples of a graph manifold carrying $k$ different Anosov flows, and includes the construction of the first author with Adam Clay \cite{ClayPinsky2025Graph} generalizing the Handel-Thurston examples to glue trefoil complements to each other. Similarly, we can prove using the same method:
    
    \begin{corollary}
     There exists a toroidal manifold with one Seifert fibered piece and one hyperbolic piece carrying $k$ different Anosov flows.
    \end{corollary}

\section{Generalized Handel--Thurston surgery achieved by Goodman--Fried surgery and consequences }

Let $T_1^2$ be a once-puntured torus with hyperbolic metric and consider the geodesic flow $\phi^t$ on its unit tangent bundle. For $k=1$ and $k=2$, consider two strongly adapted bicontact plugs $\mathcal{M}^1$ and $\mathcal{M}^2$ that support two flows that are respectively orbit equivalent to the geodesic flow $\phi^t$ and the $2$-fold cover of $\phi^t$. Gluing these plugs along their boundary yields a new strongly adapted bicontact plug on a graph manifold $\mathcal{M}^{1,2}=\mathcal{M}^1\cup \mathcal{M}^2$ that is a generalized Hendle--thurson example. In particular it is not possible to construct $\mathcal{M}^{1,2}$ by surgery along a separating simple closed geodesic on a geodesic flow on the unit tangent bundle of an hyperbolic surface of genus 2. Let  $\mathcal{M}^{1,1}=\mathcal{M}^1\cup \mathcal{M}^1$ be a geodesic flow on the unit tangent bundle of an hyperbolic surface of genus 2 realized as the union of two bicontact plugs. We show that it is possible to construct $\mathcal{M}^{1,2}$ by surgery along a family of biLegendrian knots all located in one of the two copies of $\mathcal{M}^1$.  
\begin{construction}
\label{con}
Let $T_1^2$ be a one-punctured torus with linear foliation with closed leaves directed by the vector field $V$ as constructed in Proposition \ref{GeneratingV}. Construct the $\mathcal{SAB}$-plugs $\mathcal{M}_V^1$ with $k=1$. By Remark \ref{rmk:relation to geodesic flows} this plug support a flow that is orbit equivalent to the geodesic flow on the unit tangent bundle of $T_1^2$ with hyperbolic metric. Gluing two copies of $\mathcal{M}_V^1$ along their boundary components yields to a strongly adapted bicontact structure which supported flow is orbit equivalent to the geodesic flow on the unit tangent bundle of a closed oriented surface of genus 2. Consider a simple closed curve $c_1$ transverse to the orbits of $V$ and a simple closed curve $c_2$ that is a trajectory of $V$. By thm there are isotopic curves $c_1'$ and $c_2'$ that can be realized as closed orbits for the Anosov flow $X^t$ supported by the strongly adapted bicontact structure $\mathcal{B}$ associated to $\mathcal{M}$. Since $k=1$ we have a family of 4 orbits $\{\gamma_1,\gamma_2,\gamma_3,\gamma_4\}$ pairwise isotopic $(\gamma_1,\gamma_3)$, $(\gamma_2,\gamma_4)$ and such that  and lying at different levels $w_1<w_2<w_3<w_4$. Note that the curves $c_1$ and $c_2$ form a chain in $\PMod(T_1^2)$ therefore a positive Dehn twist along $\partial T_1^2=c_3$ can be expressed as a sequence of positive Dehn twists  
$\tau_3=(\tau_1\circ \tau_2)^6=(\tau_1\circ \tau_2)^4\circ (\tau_1\circ \tau_2)^2 $.

\begin{lemma}\label{int}
Suppose that in the family $C=\{c_1,\cdots, c_r\}$ we have $c_i$ and $c_j$,  $j>i$ intersect in a single point $p$. Consider the associated biLegendrian knots $\gamma_i$ and $\gamma_j$ such that $w_j>w_i$. A sequence of surgeries $\mathcal{S}\{(\gamma_i,1),(\gamma_j,1),(\gamma_i,1)\}$ produces the change of monodromy 
$$g\rightarrow g\circ(\tau_i\circ \tau_j)^2$$
   \end{lemma}

\begin{proof}
After a simultaneous surgery along $\gamma_i$ and $\gamma_j$ the projection of the curve $\gamma_i$ on $\Sigma^0_{g,b}$ along the positive direction of the $w$ curves is obtained applying to $c_i$ a Dehn twist along $c_j$.
\end{proof}

Perform a $(1,1)$-bicontact surgery along $\gamma_1$ followed by a $(1,1)$-bicontact surgery on $\gamma_2$ we obtain $(\tau_1\circ \tau_2)$. Since $c_1\cap c_2= p$, by Proposition \ref{int} a subsequent double surgery along $\gamma_1$ induces the monodromy $(\tau_1\circ \tau_2)^4$. Finally, performing a simultaneous surgery along $\gamma_3$ and $\gamma_4$ we get $(\tau_1\circ \tau_2)^6$. Now, performing a bicontact surgery along $\gamma_1$ induces the monodromy $(\tau_1\circ \tau_2)^6$
Note that a monodromy associated to the chain relation induces a Dehn twist on $\partial T_1^2=c_3$, therefore does not change the topology of the ambient manifold that is still homeomorphic to $T_1^2\times S^1$ with natural identification of the fibers. Let $\tilde{R}_-$ be the new Reeb vector field. Note that $\tilde{R}_-=R_-$ in a collar neighborhood of the boundary. The homology class $[\tilde{w}]$ of the new Reeb orbits and the boundary orbits is now obtained by adding the homology class $[c_3]$. Thus the new projectively Anosov flow is orbit equivalent to a two-fold cover of the geodesic flow on the unit tangent bundle of a once punctured torus (see Remark \ref{rmk:relation to geodesic flows}).

\begin{remark}
The construction above can be extended to graph manifolds obtained by gluing any $k$-cover to any $h$-cover of the geodesic flow on the unit tangent bundle of a one-punctures torus.
\end{remark}

\end{construction}

Shannon \cite{Sh24} has recently shown that for some suspension Anosov flows, there exist infinitely many pairs of periodic orbits such that a non-trivial Goodman--Fried surgery produce a flow orbit equivalent to itself. 
Using the tools developed in the previous chapters we can show that in the geodesic flow of the unit tangent bundle of a hyperbolic surface of genus 2 there are sequences of closed orbits such that a non-trivial Goodman--Fried surgery produce a flow equivalent to itself. In particular we have the following.  

 \begin{corollary}
     Let $\phi^t$ be the geodesic flow on the unit tangent bundle of an hyperbolic surface of genus 2. There are infinitely many sequences of four non isotopic closed orbits $\Gamma=\{\gamma_1,\gamma_2,\gamma_3,\gamma_4\}$ such that a non-trivial Goodman--Fried surgery along $\Gamma$ produces a flow equivalent to itself.
 \end{corollary}

\begin{proof}
 Construction \ref{con} shows that performing surgery along orbits isotopic to two generators $\{c_1,c_1\}$ of the homology of $T^2_1$. For every $p,q\in \mathbb{N}$ coprime we can chose another set of curves $\{c_3,c_4\}$ on $T_1^2$ that generate the homology and such that $(\tau_3\circ \tau_4)^6=\tau_3$ where $\tau_3$ is a Dehn twist around $\partial T_1^2$. In particular $(\tau_3\circ \tau_4)^{-6}=-\tau_3$ and 
$$(\tau_1\circ \tau_2)^6\circ(\tau_3\circ \tau_4)^{-6}=id$$
Therefore a surgery on the curves associated biLegendrian curves $\Gamma=\{\gamma_1,\gamma_2,\gamma_3,\gamma_4\}$
$$\mathcal{M}^{1,1}\xrightarrow {\mathcal{S}\{(\gamma_1,q_1),\cdots,(\gamma_4,q_4)\}}\mathcal{M}^{1,1}$$
where $\mathcal{M}^{1,1}$ is the geodesic flow on the unit tangent bundle of an hyperbolic surface of genus 2.
\end{proof}

\bibliography{B}
\bibliographystyle{amsplain}

\end{document}